\documentclass[a4paper,10pt]{amsart}
\usepackage{enumerate}
\usepackage{amsthm,amsmath,amssymb,graphics,graphicx,hyperref,epstopdf,mathrsfs}
\usepackage{breqn}           
\title{Hilbert's Syzygy Theorem for monomial ideals}
\author{Guillermo Alesandroni}
\address{Department of Mathematics, Wake Forest University, 1834 Wake Forest Rd, Winston-Salem, NC 27109}
\email{alesangc@wfu.edu}

\newtheorem{theorem}{Theorem}[section]

\newtheorem{corollary}[theorem]{Corollary}

\theoremstyle{definition}
\newtheorem{definition}[theorem]{Definition}

\newtheorem{example}[theorem]{Example}
\newtheorem{construction}[theorem]{Construction}

\DeclareMathOperator{\betti}{b}
\DeclareMathOperator{\pd}{pd}
\DeclareMathOperator{\mdeg}{mdeg}
\DeclareMathOperator{\lcm}{lcm}

\DeclareMathOperator{\Max}{max}

\DeclareMathOperator{\hdeg}{hdeg}

\DeclareMathOperator{\rank}{rank}

\begin{document}
\maketitle
\begin{abstract}

\end{abstract}
We give a new proof of Hilbert's Syzygy Theorem for monomial ideals. In addition, we prove the following. If $S=k[x_1,\ldots,x_n]$ is a polynomial ring over a field, $M$ is a squarefree monomial ideal in $S$, and each minimal generator of $M$ has degree larger than $i$, then $\pd(S/M)\leq n-i$.

\section{Introduction}

Hilbert's Syzygy Theorem, first proved by David Hilbert in 1890, states that, if $k$ is a field and $M$ is a finitely generated module over the polynomial ring $S=k[x_1,\ldots,x_n]$, then the projective dimension of $M$ is at most $n$. In this article we give a simple proof of this fact for monomial ideals.

Various mathematicians have preceded us in proving Hilbert's Syzygy Theorem for finitely generated modules, or in the more restrictive context of monomial ideals. Hilbert [H], Cartan and Eilenberg [CE, Ei], Schreyer [Ei, S], and Gasharov, Peeva and Welker [GPW], have all proved some form of this result using an array of arguments, including elimination theory, Gr\"obner basis, and homological algebra. Thus, our theorem joins a generous list of similar results and yet it has a distinctive mark: our proof is simple and short. It is simple because it is constructive, intuitive, and accessible to any reader who knows the preliminaries of free resolutions. It is short because in just a few lines we prove the theorem for squarefree monomial ideals, and in a few more lines we reduce the general case to the squarefree case.

The secret to simplify and shorten our proofs is to exploit intrinsic properties of monomials such as total degree, and least common multiple. For instance, we take advantage of these properties when we introduce the Taylor resolution as a multigraded resolution. As insubstantial as it sounds, it is manipulating these basic notions that makes all the difference (Theorem \ref{Theorem squarefree} also illustrates this point).

This work comes with a bonus. In the squarefree case, we prove something slightly stronger than Hilbert's Syzygy theorem. To be precise, we show that if $M$ is a squarefree monomial ideal minimally generated by monomials of degree larger than some integer $i$, then $\pd(S/M)\leq n-i$.

 \section{Background and Notation}

Throughout this paper $S$ represents a polynomial ring over a field, in $n$ variables. The letter $M$ always denotes a monomial ideal
in $S$. 

We open this section by defining the Taylor resolution as a multigraded free resolution, something that will turn out to be fundamental in the present work. The construction that we give below can be found in [Me].
 
\begin{construction}
Let $M=(m_1,\ldots,m_q)$. For every subset $\{m_{i_1},\ldots,m_{i_s}\}$ of $\{m_1,\ldots,m_q\}$, with $1\leq i_1<\ldots<i_s\leq q$, 
we create a formal symbol $[m_{i_1},\ldots,m_{i_s}]$, called a \textbf{Taylor symbol}. The Taylor symbol associated to $\{\}$ will be denoted by $[\varnothing]$.
For each $s=0,\ldots,q$, set $F_s$ equal to the free $S$-module with basis $\{[m_{i_1},\ldots,m_{i_s}]:1\leq i_1<\ldots<i_s\leq q\}$ given by the 
${q\choose s}$ Taylor symbols corresponding to subsets of size $s$. That is, $F_s=\bigoplus\limits_{i_1<\ldots<i_s}S[m_{i_1},\ldots,m_{i_s}]$ 
(note that $F_0=S[\varnothing]$). Define
\[f_0:F_0\rightarrow S/M\]
\[s[\varnothing]\mapsto f_0(s[\varnothing])=s\]
For $s=1,\ldots,q$, let $f_s:F_s\rightarrow F_{s-1}$ be given by
\[f_s\left([m_{i_1},\ldots,m_{i_s}]\right)=
 \sum\limits_{j=1}^s\dfrac{(-1)^{j+1}\lcm(m_{i_1},\ldots,m_{i_s})}{\lcm(m_{i_1},\ldots,\widehat{m_{i_j}},\ldots,m_{i_s})}
 [m_{i_1},\ldots,\widehat{m_{i_j}},\ldots,m_{i_s}]\]
 and extended by linearity.
 The \textbf{Taylor resolution} $\mathbb{T}_M$ of $S/M$ is the exact sequence
 \[\mathbb{T}_M:0\rightarrow F_q\xrightarrow{f_q}F_{q-1}\rightarrow\cdots\rightarrow F_1\xrightarrow{f_1}F_0\xrightarrow{f_0} 
 S/M\rightarrow0.\]
 \end{construction}
Following [Me], we define the \textbf{multidegree} of a Taylor symbol $[m_{i_1},\ldots,m_{i_s}]$, denoted $\mdeg[m_{i_1},\ldots,m_{i_s}]$, as follows:
  $\mdeg[m_{i_1},\ldots,m_{i_s}]=\lcm(m_{i_1},\ldots,m_{i_s})$.

\begin{definition}

 Let $M$ be a monomial ideal, and let
 \[\mathbb{F}:\cdots\rightarrow F_i\xrightarrow{f_i}F_{i-1}\rightarrow\cdots\rightarrow F_1\xrightarrow{f_1}F_0\xrightarrow{f_0} S/M\rightarrow 0\]
be a free resolution of $S/M$. 
We say that a basis element $[\sigma]$ of $\mathbb{F}$ has \textbf{homological degree} $i$, denoted $\hdeg[\sigma]=i$, if 
$[\sigma] \in F_i$. $\mathbb{F}$ is said to be a \textbf{minimal resolution} if for every $i$, the differential matrix $\left(f_i\right)$ of $\mathbb{F}$
has no invertible entries.
\end{definition}

\textit{Note}:
From now on, every time that we make reference to a free resolution $\mathbb{F}$ of $S/M$ we will assume that $\mathbb{F}$ is obtained from $\mathbb{T}_M$ by means of consecutive cancellations. To help us remember this convention, the basis elements of a free resolution will always be called Taylor symbols.

\begin{definition}
Let $M$ be a monomial ideal, and let
 \[\mathbb{F}:\cdots\rightarrow F_i\xrightarrow{f_i}F_{i-1}\rightarrow\cdots\rightarrow F_1\xrightarrow{f_1}F_0\xrightarrow{f_0} S/M\rightarrow 0\]
be a minimal free resolution of $S/M$.
\begin{itemize}
 \item For every $i$, the $i^{th}$ \textbf{Betti number} $\betti_i\left(S/M\right)$ of $S/M$ is $\betti_i\left(S/M\right)=\rank(F_i)$.
\item For every $i,j\geq 0$, the \textbf{graded Betti number} $\betti_{ij}\left(S/M\right)$ of $S/M$, in homological degree $i$ and internal degree $j$,
is \[\betti_{ij}\left(S/M\right)=\#\{\text{Taylor symbols }[\sigma]\text{ of }F_i:\deg[\sigma]=j\}.\]
\item For every $i\geq 0$, and every monomial $l$, the \textbf{multigraded Betti number} $\betti_{i,l}\left(S/M\right)$ of $S/M$, in homological degree $i$ and multidegree $l$,
is \[\betti_{i,l}\left(S/M\right)=\#\{\text{Taylor symbols }[\sigma]\text{ of }F_i:\mdeg[\sigma]=l\}.\]
\item The \textbf{projective dimension} $\pd\left(S/M\right)$ of $S/M$ is \[\pd\left(S/M\right)=\max\{i:\betti_i\left(S/M\right)\neq 0\}.\]
\end{itemize}
\end{definition}

\section{Hilbert's Syzygy Theorem in the squarefree case}

Without preamble, we state and prove one of our main results.

\begin{theorem}\label{Theorem squarefree}
Let $M=(m_1,\ldots,m_q)$ be a squarefree monomial ideal. Suppose that $\deg (m_1),\ldots,\deg (m_q) > k$, for some $k\geq 0$. Then $\pd(S/M)\leq n-k$.
\end{theorem}

\begin{proof}
Let 
\[\mathbb{F}: 0\rightarrow F_p\xrightarrow{f_p} F_{p-1}\cdots F_1\xrightarrow{f_1} F_0 \xrightarrow{f_0} S/M \rightarrow 0\]
 be a minimal resolution of $S/M$. Let $[\theta]$ be a Taylor symbol of $F_p$ and let $f_p[\theta]=\sum a_i [\tau_i]$. By the minimality of  $\mathbb{F}$, none of the $a_i$ is invertible, and at least one of the $a_i$ is not zero, say $a_r\neq 0$. It follows that 
 $\mdeg[\theta]=\mdeg(a_r[\tau_r]) $. Let $[\sigma_p]=[\theta]$ and $[\sigma_{p-1}]=[\tau_r]$. Note that $\deg[\sigma_{p-1}]<\deg[\sigma_p]$.\\
Suppose that $[\sigma_p],\ldots,[\sigma_{p-j}]$ are Taylor symbols of $F_p,\ldots, F_{p-j}$, respectively, such that, for all $i=1,\ldots,j$, $\deg[\sigma_{p-i}]<\deg[\sigma_{p-i+1}]$.\\ 
 Let $f_{p-j}[\sigma_{p-j}]=\sum \betti_i [\xi_i]$. By the minimality of $\mathbb{F}$, none of the $\betti_i$ is invertible, and at least one of the $\betti_i$ is not zero, say $\betti_s\neq 0$. It follows that $\mdeg[\sigma_{p-j}]=\mdeg(\betti_s[\xi_s])$.\\
 Let $[\sigma_{p-j-1}]=[\xi_s]$. Note that $\deg[\sigma_{p-j-1}]<\deg[\sigma_{p-j}]$. Thus, we can recursively define a sequence $[\sigma_1],\ldots,[\sigma_p]$ of Taylor symbols of $F_1,\ldots,F_p$, respectively, such that $k+1\leq \deg[\sigma_1]<\ldots<\deg[\sigma_p] \leq n$. Thus, $\{\deg[\sigma_1],\ldots,\deg[\sigma_p]\}$ is a subset of $\{k+1,\ldots,n\}$, and hence, $p=\#\{\deg[\sigma_1],\ldots,\deg[\sigma_p]\}\leq \#\{k+1,\ldots,n\}= n-k$.
 \end{proof}

Theorem \ref{Theorem squarefree} has some interesting applications. For instance, if $M$ is an edge ideal, then $\pd(S/M)\leq n-1$. More importantly, Hilbert's Syzygy Theorem for squarefree monomial ideals follows from Theorem \ref{Theorem squarefree}, with $k=0$.

\section{Hilbert's Syzygy Theorem for monomial ideals}

The following theorem is due to Gasharov, Hibi, and Peeva [GHP, Theorem 2.1]. We change their terminology and notation to make it consistent with our own.

\begin{theorem}  \label{Theorem GHP} 
Let $M$ be minimally generated by $G$, and let $\mathbb{F}$ be a minimal resolution of $S/M$. Given a monomial $m$, consider the ideal $M_m$, minimally generated by the elements of $G$ dividing $m$. Denote by 
$\mathbb{F}_m$ the subcomplex of $\mathbb{F}$ generated by the Taylor symbols of $\mathbb{F}$ with multidegrees dividing $m$. Then $\mathbb{F}_m$ is a minimal resolution of $S/M_m$.
\end{theorem}

\begin{corollary}\label{Corollary 2}
Let $M$ be minimally generated by $G$. Given a monomial $m$, consider the ideal $M_m$ minimally generated by the elements of $G$ dividing $m$. Then, for all $i$
\[\betti_{i,m}(S/M)=\betti_{i,m}(S/M_m).\]
\end{corollary}

\begin{proof}
Let $\mathbb{F}$ be a minimal resolution of $M$, and let $\mathbb{F}_m$ be the minimal resolution of $S/M_m$, given by Theorem \ref{Theorem GHP}. By construction, $\mathbb{F}$ and $\mathbb{F}_m$ have the same Taylor symbols $[\sigma]$, with $\hdeg[\sigma]=i$, and $\mdeg[\sigma]=m$. 
\end{proof}

\begin{construction}
Let $M=(m_1,\ldots,m_q)$, where $m_i=x_1^{\alpha_{i1}}\ldots x_n^{\alpha_{in}}$, for all $i$. Let $m=\lcm(m_1,\ldots,m_q)$. Then $m$ factors as $m=x_1^{\alpha_1}\ldots x_n^{\alpha_n}$, where $\alpha_j=\Max(\alpha_{1j},\ldots,\alpha_{qj})$, for all $j=1,\ldots,n$. For all $i=1,\ldots,q$, define $m'_i=x_1^{\beta_{i1}}\ldots x_n^{\beta_{in}}$, where 
\[\beta_{ij}=
\begin{cases}
\alpha_j\text{, if } \alpha_{ij}=\alpha_j \\
0\text{, otherwise.}
\end{cases}\]
Let $M'=(m'_1,\ldots,m'_q)$. The ideal $M'$ will be referred to as the \textbf{twin ideal} of $M$.
\end{construction}

\begin{example}
Let $M=(m_1=x^2y^2z, m_2=x^2z^2, m_3=y z^2)$. \\
Notice that $m=\lcm(m_1,m_2,m_3)
=x^2y^2z^2$. Then $m'_1=x^2z^2$; $m'_2=x^2z^2$; $m'_3=z^2$. Thus, the twin ideal of $M$ is $M'=(x^2y^2,x^2z^2,z^2)=(x^2y^2,z^2)$.
\end{example}

The terminology twin ideal is suggestive; the next theorem [Al, Theorem 4.10] justifies this choice of words. 

\begin{theorem} \label{Al Theorem 4.10}
Let $M'$ be the twin ideal of $M=(m_1,\ldots,m_q)$, and let $m=\lcm(m_1,\ldots,m_q)$. Then $\betti_{i,m}(S/M)=\betti_{i,m}(S/M')$, for all $i$.
\end{theorem}

Now we have all the tools to prove Hilbert's Syzygy theorem for monomial ideals.
 
 \begin{theorem}
Let $M$ be a monomial ideal of $S=k[x_1,\ldots,x_n]$. Then $\pd(S/M)\leq n$.
 \end{theorem}
 
 \begin{proof}
 Let $l$ be a monomial. Let $M_l$ be the ideal generated by the elements of the minimal generating set of $M$ dividing $l$. Let us denote by $\{m_1,\ldots,m_q\}$ the minimal generating set of $M_l$, where 
 $m_i=x_1^{\alpha_{i1}}\ldots x_n^{\alpha_{in}}$, for all $i$. Let $\lcm(m_1,\ldots,m_q)=x_1^{\alpha_1}\ldots x_n^{\alpha_n}$. Then the twin ideal of $M_l$ is given by $M'_l=(m'_1,\ldots,m'_q)$, where each $m'_i$ factors as 
\[m'_i=x_1^{\beta_{i1}}\ldots x_n^{\beta_{in}} \text{, with } \beta_{ij}=
 \begin{cases}
 \alpha_j \text{, if }\alpha_{ij}=\alpha_j \\
 0 \text{, otherwise.}
 \end{cases}
\]
 In particular, each $x_j$ appears with exponent either $\alpha_j$ or $0$ in the factorization of each generator $m'_1,\ldots,m'_q$. Let us make the change of variables $y_1=x_1^{\alpha_1},\ldots,y_n=x_n^{\alpha_n}$. Then
 $m'_1,\ldots,m'_q$ can be represented in the form
 
\[ \begin{array}{ccl}
 m'_1 &= &y_1^{\delta_{11}} \ldots y_n^{\delta_{1n}}\\
        &\vdots & \\
 m'_q & =  & y_1^{\delta_{q1}}  \ldots y_n^{\delta_{qn}},      
 \end{array}\]
 where, for all $1\leq i\leq q$, and all $1\leq j\leq n$, $\delta_{ij}=\begin{cases}
 														1\text{ if } \alpha_{ij}=\alpha_j\\
														0 \text{ if } \alpha_{ij}\neq \alpha_j.
														\end{cases}$
Hence, we can interpret $M'_l$ as a squarefree monomial ideal in $k[y_1,\ldots,y_n]$. By Theorem \ref{Theorem squarefree}, $\pd\left(\dfrac{k[y_1,\ldots,y_n]}{M'_l}\right)\leq n$. Therefore, $\pd (S/M'_l) \leq n$ and thus, $\betti_{i,l}(S/M'_l)=0$, for all $i\geq n+1$. Finally, from Corollary \ref{Corollary 2} and Theorem \ref{Al Theorem 4.10}, we obtain that $\betti_{i,l}(S/M)=\betti_{i,l}(S/M_l)=\betti_{i,l}(S/M'_l)=0$ for all $i\geq n+1$. Since $l$ is an arbitrary monomial, we must have that $\pd(S/M)\leq n$.
 \end{proof}

 \bigskip

\noindent \textbf{Acknowledgements}: A big thanks to my friend Mauricio Rivas for his valuable comments. I am grateful to my dear wife Danisa for her support and encouragement, and for typing this article.

\end{document}